\documentclass{amsart}

\usepackage[mathscr]{eucal}
\usepackage{amssymb}
\usepackage[usenames,dvipsnames]{xcolor} 
\usepackage[normalem]{ulem}
\usepackage{amsthm}
\usepackage{mathtools}
\usepackage{bbold}
\usepackage{mathdots}
\usepackage{float}
\usepackage{enumitem}

\usepackage{adjustbox}
\usepackage{amsmath}
\usepackage{wasysym}
\usepackage{tikz}
\usetikzlibrary{calc}
\usepackage{tikz-cd}
\usetikzlibrary{arrows}
\usetikzlibrary{shapes}
\usetikzlibrary{patterns.meta}
\usetikzlibrary{snakes}

\usepackage[unicode]{hyperref} 

\hypersetup{colorlinks=true,citecolor=Brown,urlcolor=Brown,linkcolor=Brown}

\usepackage{microtype}
\usepackage{aliascnt}
\usepackage[nameinlink,capitalise,noabbrev]{cleveref}

\numberwithin{equation}{section}
\newcommand{\declareeqenv}[3]{%
  \newaliascnt{#1}{equation}%
  \newtheorem{#1}[#1]{#2}%
  \aliascntresetthe{#1}%
  \crefname{#1}{#2}{#3}%
  \Crefname{#1}{#2}{#3}%
}

\declareeqenv{Thm}{Theorem}{Theorems}
\newtheorem*{Thm*}{Theorem}
\declareeqenv{Prop}{Proposition}{Propositions}
\declareeqenv{Lem}{Lemma}{Lemmas}
\declareeqenv{Cor}{Corollary}{Corollarys}
\declareeqenv{Conj}{Conjecture}{Conjectures}
\declareeqenv{Claim}{Claim}{Claims}
\declareeqenv{Hope}{Hope}{Hopes}

\theoremstyle{remark}
\declareeqenv{Def}{Definition}{Definitions}
\declareeqenv{Ter}{Terminology}{Terminologys}
\declareeqenv{Not}{Notation}{Notations}
\declareeqenv{Exa}{Example}{Examples}
\declareeqenv{Exas}{Examples}{Exampless}
\declareeqenv{Cons}{Construction}{Constructions}
\declareeqenv{Conv}{Convention}{Conventions}
\declareeqenv{Hyp}{Hypothesis}{Hypothesiss}

\declareeqenv{Rem}{Remark}{Remarks}
\declareeqenv{Rec}{Recollection}{Recollections}
\declareeqenv{Que}{Question}{Questions}

\newcommand{\nc}{\newcommand}
\nc{\dmo}{\DeclareMathOperator}

\usepackage{todonotes}

\nc{\MonSeq}{\mathsf{MSeq}}
\nc{\AsymSeq}{\mathsf{ASeq}}
\nc{\PowerSeq}{\mathsf{PSeq}}
\nc{\sigmakf}{\sigma^k\hspace{-.25ex}f}
\nc{\sigmakg}{\sigma^k\hspace{-.25ex}g}
\nc{\sigmakh}{\sigma^k\hspace{-.25ex}h}
\nc{\sigmainff}{\sigma^{\inf(f)}\hspace{-.25ex}f}
\nc{\dotimes}{\otimes^{\bbL}}
\dmo{\QCoh}{QCoh}
\dmo{\Nil}{Nil}
\dmo{\coker}{coker}
\dmo{\nil}{nil}
\nc{\PRad}{\mathsf{PRad}}
\nc{\mono}[1]{\widehat{#1}}
\nc{\radideal}[1]{\sqrt{\langle #1 \rangle}}
\nc{\ideal}[1]{\langle #1 \rangle}
\dmo{\thickiddmo}{thickid}
\dmo{\Locdmo}{Loc}
\dmo{\Lociddmo}{Locid}
\dmo{\thickdmo}{thick}
\nc{\thicksub}[1]{\thickdmo(#1)}
\nc{\thickidsub}[1]{\thickiddmo(#1)}
\nc{\Locsub}[1]{\Locdmo(#1)}
\nc{\Locidsub}[1]{\Lociddmo(#1)}
\nc{\loew}{{\ell\ell}}
\nc{\tame}{\tau}
\nc{\Specplus}{\Spec_+\hspace{-0.1em}}
\nc{\SpecRK}{\Spec(R_{\cat K})}
\nc{\altmathbb}[1]{\mathbbold{#1}}
\nc{\gP}{\altmathbb{g}_{\cat P}}
\nc{\gp}{g_{\frakp}}
\nc{\Pone}{\mathbb{P}^1_k}
\nc{\PnY}{\mathbb{P}^n_Y}

\nc{\SpcT}{\Spc(\cat T^c)}
\nc{\SpcS}{\Spc(\cat S^c)}
\nc{\frakm}{\mathfrak m}
\nc{\overbar}[1]{\mkern 1.5mu\overline{\mkern-1.5mu#1\mkern-1.5mu}\mkern 1.5mu}
\nc{\bbullet}{{\scriptscriptstyle\hspace{-1pt}\bullet}}
\nc{\bullett}{{\scriptscriptstyle\bullet}\hspace{-1pt}}
\nc{\LF}{L\hspace{-0.2ex}F}
\nc{\SpG}{\Sp^G}
\nc{\Prst}{{\cat P}\mathrm{r^{st}}}
\nc{\Mack}{\mathcal{M}ack}
\nc{\SC}{S\cat C}
\nc{\OY}{\cat O_{\hspace{-0.2ex}Y}}
\nc{\OYy}{\cat O_{\hspace{-0.2ex}Y,y}}
\nc{\OYfx}{\cat O_{\hspace{-0.2ex}Y,f(x)}}
\nc{\OX}{\cat O_{\hspace{-0.2ex}X}}
\nc{\OXX}{\OX\hspace{-0.1ex}(X)}
\nc{\OXx}{\cat O_{\hspace{-0.2ex}X,x}}
\nc{\OXy}{\cat O_{\hspace{-0.2ex}X,y}}
\dmo{\BLat}{\mathsf{BLat}}
\dmo{\BDLat}{\mathsf{BDLat}}
\dmo{\DLat}{\mathsf{DLat}}
\dmo{\Lat}{\mathsf{Lat}}
\dmo{\Pos}{\mathsf{Pos}}
\dmo{\Spectral}{\mathsf{Spec}}
\dmo{\kosz}{Kosz}
\nc{\Kosz}{\kosz}
\dmo{\Ann}{Ann}
\dmo{\Stab}{Stab}
\dmo{\Aff}{Aff}
\dmo{\Ext}{Ext}
\dmo{\Tor}{Tor}
\dmo{\gen}{gen}
\dmo{\Pic}{Pic}
\dmo{\DM}{DM}
\dmo{\DMT}{DMT}
\dmo{\DMAT}{DMAT}
\nc{\DMQ}{\DM_Q}
\dmo{\DerKal}{DMack}
\dmo{\Der}{D}
\dmo{\Dperf}{\Der_\mathrm{perf}}
\dmo{\Dps}{\Der_\mathrm{pcoh}}
\nc{\Derps}{\Dps}
\nc{\Dpfl}{\Der_{+}^{\mathrm{fl}}}
\nc{\Dzfl}{\Der_{\ge 0}^{\mathrm{fl}}}
\nc{\Dpfg}{\Der_{+}^{\mathrm{fg}}}
\dmo{\Dbfl}{\Der_{b}^{\mathrm{fl}}}
\dmo{\Dbfg}{\Der_{b}^{\mathrm{fg}}}
\dmo{\Dbcoh}{\Der^{b}_{\hspace{-0.2ex}\mathrm{coh}}}
\dmo{\Dpcoh}{{\Der}^{{\sqsupset}}_{\mathrm{coh}}}
\dmo{\Dminus}{{\Der}^{{\sqsubset}}}
\nc{\Esplit}{E_{\mathrm{split}}}
\nc{\Eeven}{E_{\mathrm{even}}}
\nc{\Eodd}{E_{\mathrm{odd}}}
\nc{\fsplit}{f_{\mathrm{split}}}
\dmo{\Ddash}{\Der^{\mathrm{--}}}
\dmo{\Ddashfl}{\Der^{\mathrm{--}}_{\mathrm{fl}}}
\nc{\Dmfl}{\Ddashfl}
\nc{\Derqc}{\Der_{\mathrm{qc}}}
\dmo{\DMot}{DMot}
\dmo{\rmH}{H}
\dmo{\piu}{\underline{\pi}}
\dmo{\Sphere}{\mathbb{S}}
\nc{\HA}{{\rmH \hspace{-0.2em}\bbA}}
\nc{\HZ}{{\rmH \hspace{-0.2em}\bbZ}}
\nc{\HZp}{{\rmH \hspace{-0.2em}\bbZ_{(p)}}}
\nc{\HZbar}{{\rmH \hspace{-0.2em}\underline{\bbZ}}}
\nc{\Fp}{{\bbF_{\hspace{-0.1em}p}}}
\nc{\HFp}{{\rmH \hspace{-0.15em}\bbF_{\hspace{-0.1em}p}}}
\nc{\DHZpG}{\Der(\HZp_G)}
\nc{\DHZG}{\Der(\HZ_G)}
\nc{\DHZH}{\Der(\HZ_H)}
\nc{\DHZK}{\Der(\HZ_K)}
\nc{\DHZGN}{\Der(\HZ_{G/N})}
\nc{\DHZGG}{\Der(\HZ_{G/G})}
\nc{\DHZCp}{\Der(\HZ_{C_p})}
\nc{\DHZGprime}{\Der(\HZ_{G'})}
\nc{\DHZ}{\Der(\HZ)}
\nc{\frakp}{\mathfrak{p}}
\nc{\frakq}{\mathfrak{q}}
\nc{\Z}{\mathbb{Z}}
\nc{\SSG}{\text{sSet}_*^G}
\nc{\sSet}{\text{sSet}}

\nc{\Loco}[1]{\Loc_{\otimes}\hspace{-0.3ex}\langle #1 \rangle}
\nc{\Coloco}[1]{\Coloc^{\mathrm{Hom}}\hspace{-0.3ex}\langle #1 \rangle}

\dmo{\Con}{Conj}
\dmo{\Sub}{Sub}
\dmo{\Id}{Id}
\dmo{\Loc}{Loc}
\dmo{\Thick}{Thick}
\dmo{\Coloc}{Coloc}
\dmo{\rmK}{\textrm{\rm K}}
\dmo{\Spc}{Spc}
\dmo{\cone}{cone}
\dmo{\End}{End}
\dmo{\Mor}{Mor}
\dmo{\Hom}{Hom}
\dmo{\id}{id}
\dmo{\incl}{incl}
\dmo{\Img}{Im}
\dmo{\im}{im}
\dmo{\Ker}{Ker}
\dmo{\ind}{ind}
\dmo{\CoInd}{coind}
\dmo{\res}{res}
\dmo{\infl}{infl}
\dmo{\triv}{triv}
\dmo{\Tel}{Tel} 
\dmo{\Mod}{Mod}%
\dmo{\opname}{op}
\dmo{\SH}{SH}
\dmo{\smallb}{b}
\dmo{\Spec}{Spec}
\dmo{\supp}{supp}
\dmo{\Supp}{Supp}
\dmo{\Cosupp}{Cosupp}
\nc{\SHc}{{\SH^c}}
\nc{\SHp}{{\SH_{(p)}}}
\nc{\SHcp}{{\SH^c_{(p)}}}
\nc{\SHG}{\SH(G)}
\nc{\SHGp}{\SH(G)_{(p)}}
\nc{\SHGc}{\SHG^c}
\nc{\SHGcp}{\SHG^c_{(p)}}
\nc{\quadtext}[1]{\quad\textrm{#1}\quad}
\nc{\qquadtext}[1]{\qquad\textrm{#1}\qquad}
\nc{\adj}{\dashv}
\nc{\adjto}{\rightleftarrows}
\nc{\bbL}{\mathbb{L}}
\nc{\bbA}{\mathbb{A}}
\nc{\bbN}{\mathbb{N}}
\nc{\bbQ}{\mathbb{Q}}
\nc{\bbZ}{\mathbb{Z}}
\nc{\bbF}{\mathbb{F}}
\nc{\bbR}{\mathbb{R}}
\nc{\cat}[1]{\mathscr{#1}}
\nc{\ie}{{\sl i.e.}, }
\nc{\into}{\mathop{\rightarrowtail}}
\nc{\inv}{^{-1}}
\nc{\isoto}{\mathop{\overset{\sim}\to}}
\nc{\isotoo}{\mathop{\overset{\sim}\too}}
\nc{\onto}{\mathop{\twoheadrightarrow}}
\nc{\too}{\mathop{\longrightarrow}\limits}
\nc{\mapstoo}{\longmapsto}
\nc{\adh}[1]{\overline{#1}}
\nc{\adhpt}[1]{\adh{\{#1\}}}
\nc{\aka}{{a.\,k.\,a.}\ }
\nc{\calF}{\mathcal{F}}
\nc{\eg}{{\sl e.\,g.}}
\nc{\Homcat}[1]{\Hom_{\cat #1}}
\nc{\hook}{\hookrightarrow}
\nc{\ihomname}{\mathsf{hom}}
\nc{\ihom}[1]{\mathsf{hom}(#1)}
\nc{\Mid}{\,\big|\,}
\nc{\MMod}{\,\text{-}\Mod}%
\nc{\op}{^{\opname}}
\nc{\oto}[1]{\overset{#1}\to}
\nc{\otoo}[1]{\overset{#1}{\,\too\,}}
\nc{\sminus}{\!\smallsetminus\!}
\nc{\poplus}[1]{^{\oplus #1}}%
\nc{\potimes}[1]{^{\otimes #1}}
\nc{\sbull}{{\scriptscriptstyle\bullet}}
\nc{\SET}[2]{\big\{\,#1\Mid#2\,\big\}}
\nc{\SpcK}{\Spc(\cat K)}
\nc{\then}{\Rightarrow}
\nc{\unit}{\mathbb{1}}
\nc{\unitT}{\unit_{\cat T}}
\nc{\unitS}{\unit_{\cat S}}
\nc{\xra}{\xrightarrow}
\nc{\phigeom}[1]{\widetilde{\Phi}^{#1}}
\nc{\phigeomb}[1]{\Phi^{#1}}
\dmo{\Oname}{O}
\dmo{\proper}{proper}
\dmo{\lenormal}{\unlhd}
\dmo{\lnormal}{\lhd}
\nc{\normal}{\trianglelefteq}
\nc{\Op}{\Oname^p}
\nc{\Oq}{\Oname^q}
\dmo{\Sp}{Sp}
\dmo{\Ho}{Ho}
\dmo{\Fin}{Fin}
\dmo{\add}{add}
\dmo{\Fun}{Fun}
\dmo{\CAlg}{CAlg}
\dmo{\CMon}{CMon}
\dmo{\CC}{\cat C}
\dmo{\DD}{\cat D}
\dmo{\OO}{\mathcal{O}}
\dmo{\Map}{Map}
\dmo{\Span}{Span}
\dmo{\N}{N}
\dmo{\Cat}{Cat}
\dmo{\colim}{colim}
\dmo{\Ch}{Ch}
\dmo{\A}{\mathbb{A}^{eff}}
\nc{\AGeff}{\mathbb{A}_G^{\mathrm{eff}}}
\nc{\BGeff}{\mathcal{B}_G^{\mathrm{eff}}}
\nc{\BG}{{\mathcal{B}_G}}
\nc{\NBGeff}{{\N}{\BGeff}}
\dmo{\Ab}{Ab}
\nc{\Set}{\mathsf{Set}}
\dmo{\ev}{ev}
\dmo{\Spcl}{Spcl}
\nc{\Funadd}{\Fun_{\add}}
\dmo{\proj}{proj}
\dmo{\cof}{cof}

\newcounter{enum-resume-hack}

\begin{document}


\title[The Serre--Grothendieck finiteness theorem]{\vspace*{-1.5em}The Serre--Grothendieck finiteness theorem for the cohomology of coherent sheaves}
\author{Beren Sanders}
\date{September 23, 2026}

\address{Beren Sanders, Mathematics Department, UC Santa Cruz, 95064 CA, USA}
\email{beren@ucsc.edu}
\urladdr{http://people.ucsc.edu/$\sim$beren/}

\maketitle

\begin{abstract}
\vspace*{-2.0em}
	We prove the converse of the Serre--Grothendieck finiteness theorem for the cohomology of coherent sheaves, originally due to Lipman. This result characterizes proper morphisms as those morphisms whose derived pushforward preserves pseudo-coherent complexes. Although the derived pushforward does not reflect pseudo-coherence, we prove that it does after twisting by all perfect complexes. This provides a new characterization of properness.
\end{abstract}

\vspace{2em}



In his seminal work on sheaf cohomology in algebraic geometry, Serre proved the following fundamental finiteness result:

\begin{Thm}[Serre]
	Let $X$ be a projective variety and let $\cat F$ be a coherent sheaf on $X$. The sheaf cohomology group $H^i(X;\cat F)$ is finite-dimensional for each $i$ and vanishes for $i \gg 0$.
\end{Thm}

This is \cite[Théorème~1, p.~259]{Serre55_fac}. He then conjectures that the theorem should also be true for every complete variety. This was soon proved by Grothendieck \cite[Th\'{e}or\`{e}me~4]{Grothendieck58} and included in \cite[Th\'{e}or\`{e}me~3.2.1]{EGA3a}. Ultimately, with the development of derived categories, Grothendieck established a vast generalization of Serre's theorem:

\begin{Thm}[Grothendieck]\label{thm:grothendieck}
	Let $f:X \to Y$ be a proper morphism of noetherian schemes. The derived pushforward $f_*:\Der(X) \to \Der(Y)$ preserves pseudo-coherent complexes.
\end{Thm}

In this noetherian context, pseudo-coherent complexes are simply the bounded above complexes of coherent sheaves. See~\cref{sec:preliminaries} for further details.

For many years it was not known --- or, at least, not widely known --- that the converse of this Serre--Grothendieck finiteness theorem is also true. In other words, preservation of pseudo-coherence actually characterizes properness. This result is due to Lipman; a sketch of a proof is given in \mbox{\cite[(4.3.9)]{Lipman09}}. In this paper, we provide a proof of Lipman's result and augment it with a new characterization of properness. More precisely, we prove the following:

\begin{Thm}\label{thm:main}
	Let $f:X \to Y$ be a finite-type separated morphism of noetherian schemes and let $f_*:\Der(X)\to \Der(Y)$ denote the derived pushforward. The following are equivalent:
	\begin{enumerate}
		\item $f$ is a proper morphism;
		\item $f_*$ preserves pseudo-coherent complexes;
		\item $\Dps(X) = \SET{a \in \Der(X)}{f_*(a \otimes b)\in \Dps(Y) \text{ for all } b\in \Dperf(X)}$.
	\end{enumerate}
\end{Thm}
\noindent
Here $\Dps(X)$ denotes the category of pseudo-coherent complexes and~$\Dperf(X)$ denotes the category of perfect complexes. There are also equivalent versions of criteria $(b)$ and $(c)$ formulated in terms of bounded coherent complexes rather than pseudo-coherent complexes; see \cref{thm:main-body} and \cref{thm:dbcoh-version}. The new characterization $(c)$ is delicate: it fails if we remove the twisting by the object~$b$. In other words, the derived pushforward of a proper morphism does not reflect pseudo-coherence in general; see \cref{exa:no-reflect}.

The proof of \cref{thm:main} combines classical algebro-geometric tools --- specifically, Nagata's Compactification Theorem and Beilinson's Resolution of the Diagonal --- with general tools from tensor triangular geometry. In particular, we study the Balmer spectrum of the derived category of pseudo-coherent complexes and prove that for any quasi-compact and quasi-separated scheme~$X$, the fibers of the canonical map \begin{equation}\label{eq:intro-canonical} \Spc(\Dps(X)) \to X \end{equation} are irreducible. See~\cref{thm:fibers-irreducible}.

The irreducibility of these fibers provides a clean way to study pseudo-coherence under pullbacks. In general, the thick subcategory generated by a pseudo-coherent complex need not contain any nonzero perfect complex --- that is, a pseudo-coherent complex need not be ``virtually small'' in the sense of~\cite{DwyerGreenleesIyengar06}. However, this fiber property implies that any pseudo-coherent complex can be ``twisted'' into a virtually small complex; see~\cref{prop:virtually-small}. This, in turn, yields an effortless proof that the derived pullback of a surjective proper morphism reflects pseudo-coherence; see \cref{prop:reflects-pseudo}. This is false for bounded coherent complexes, which again highlights the central role played by pseudo-coherence.

Overall, our approach to~\cref{thm:main} prioritizes general categorical and tensor triangular reasoning whenever applicable.
\[\ast\ast\ast\]
The author has learned that the implication $(a)\Rightarrow (c)$ in~\cref{thm:main} has also been obtained by Rossanigo~\cite[Corollary~1.0.4]{Rossanigo26pp} through the lens of representability theorems for functors defined on $\Dperf(X)$.

\section{Preliminaries}\label{sec:preliminaries}
We begin by setting some notation and recollecting some basic facts.

\begin{Not}
	For a collection of objects $\cat E$ in a tensor-triangulated category $\cat K$, we will write $\thickidsub{\cat E}$ for the thick tensor ideal generated by $\cat E$. A standard thick subcategory argument shows that 
	\begin{equation*}
		\thickidsub{\cat E} = \thicksub{\cat E \otimes \cat K}.
	\end{equation*}
\end{Not}

\begin{Rec}\label{rec:nilpotence}
	For any morphism $\eta: \unit \to b$ in a tensor-triangulated category $\cat K$, we have an associated exact triangle
	\begin{equation}\label{eq:nilpotence-triangle}
		w \xrightarrow{\xi} \unit \xrightarrow{\eta} b \to \Sigma w
	\end{equation}
	and $\Nil(\xi) \coloneqq \SET{a \in \cat K}{\xi^{\otimes n} \otimes a = 0 \text{ for some } n \ge 1}$ is a thick tensor ideal. This is~\cite[Proposition~2.9]{Balmer18}. 
\end{Rec}

\begin{Conv}\label{conv:qcqs}
	All of our schemes $X$ will be quasi-compact and quasi-separated. We will sometimes assume $X$ is noetherian, but will always be explicit about this hypothesis.
\end{Conv}

\begin{Rem}
	We will write $\Der(X) \coloneqq \Derqc(X)$ for the derived category of complexes of $\OX$-modules whose cohomology modules are quasi-coherent. It is a rigidly-compactly generated tensor-triangulated category whose compact objects (=dualizable objects) are precisely the perfect complexes. These form a tensor-triangulated subcategory
	\[
		\Dperf(X) \subset \Der(X).
	\]
	The original scheme can be recovered as the Balmer spectrum of this category: $X \cong \Spc(\Dperf(X))$. Under this identification, a point $x \in X$ corresponds to the Balmer prime $\cat P_x \coloneqq \Ker(\Dperf(X) \to \Dperf(\OXx))$. See~\cite{BalmerICM}. We will simply write~$\otimes$ for the derived tensor product in $\Der(X)$ and, when convenient, write $\unit_X= \OX$ for the tensor unit. The internal hom will be denoted~$\ihomname$.
\end{Rem}

\begin{Not}
	For any morphism $f:X\to Y$, we will write $f^*:\Der(Y)\to\Der(X)$ for the derived pullback and $f_*:\Der(X)\to\Der(Y)$ for the derived pushforward. Moreover, we will write $f^!:\Der(Y)\to\Der(X)$ for the right adjoint of $f_*$. These adjoints are related by numerous formulas, detailed in \cite{BalmerDellAmbrogioSanders16}. In particular, we have the projection formula: $f_*(a\otimes f^*(b)) \simeq f_*(a)\otimes b$ for any $a \in \Der(X)$ and $b \in \Der(Y)$. Since~$f^*$ is (strong) symmetric monoidal, it preserves perfect complexes, simply because they are the dualizable objects.
\end{Not}

\begin{Rec}\label{rec:bounded}
	The derived pullback $f^*$ is right $t$-exact and the derived pushforward~$f_*$ is left $t$-exact, meaning $f^*(\Der^{\le 0}(Y)) \subseteq \Der^{\le 0}(X)$ and $f_*(\Der^{\ge 0}(X)) \subseteq \Der^{\ge 0}(Y)$. Although $f_*$ is not right $t$-exact (except when $f$ is affine), it is always right \mbox{$t$-bounded} meaning $f_*(\Der^{\le 0}(X)) \subseteq \Der^{\le d}(Y)$ for some fixed $d \ge 0$. See \cite[Proposition~3.9.2]{Lipman09}. In contrast, $f^*$ is left \mbox{$t$-bounded} only for very special morphisms (namely, those of finite Tor-dimension). 
\end{Rec}

\begin{Rem}
	We write $\Dbcoh(X)$ and $\Dps(X)$ for the categories of bounded coherent complexes and pseudo-coherent complexes, respectively. While the general definition of pseudo-coherent complex is somewhat technical, it will suffice for our purposes to understand the affine case and the noetherian case, which we recollect below. The general definition is discussed in~\cite[Section~2]{SGA6-expose-I}, \cite[Section~2]{ThomasonTrobaugh90} and \cite[Sections~\href{https://stacks.math.columbia.edu/tag/064N}{064N},~\href{https://stacks.math.columbia.edu/tag/08CA}{08CA} and~\href{https://stacks.math.columbia.edu/tag/08E4}{08E4}]{stacks-project}.
\end{Rem}

\begin{Rem}\label{rem:affine-pseudo}
	Let $R$ be a commutative ring. A complex $E\in \Der(R)$ is pseudo-coherent if it is quasi-isomorphic to a complex $P$ of finitely generated projective modules that is strictly bounded on the right. Moreover, by starting at the right and successively removing contractible summands, we can always replace~$P$ by a homotopy equivalent complex that satisfies
	\begin{equation}\label{eq:top-homology}
		\max\SET{i}{P^i \neq 0} = \max\SET{i}{ H^i(P) \neq 0}\eqqcolon m.
	\end{equation}
	This makes it clear that the rightmost cohomology group
	\[
		H^m(E) = \coker(P^{m-1}\to P^{m})
	\]
	is finitely presented. Over a non-noetherian ring, $H^i(E)$ need not be finitely generated for $i < m$.
\end{Rem}

\begin{Exa}
	Let $R=k \ltimes V$ be the trivial square-zero extension of a field~$k$ by an infinite-dimensional $k$-vector space $V$. This ring is not coherent by \mbox{\cite[Thm.~2.6]{KabbajMahdou04}} and for any nonzero $v \in V$, the perfect complex
	\[
		P=(R \xrightarrow{ (0,v) } R)
	\]
	has an infinitely generated $H^{-1}(P) = 0 \ltimes V$.
\end{Exa}

\begin{Rem}\label{rem:pseudo-nak}
	If $(R,\frakm)$ is a local ring then we can always choose the complex~$P$ representing~$E$ to satisfy $d^i(P^i) \subseteq \frakm P^{i+1}$ for all $i$. This is achieved by starting at the right and successively removing contractible summands ${(\id:R \to R)}$ associated to units in the matrices representing the differentials; cf.~\cite[Lemma~\href{https://stacks.math.columbia.edu/tag/00MT}{00MT}]{stacks-project}. It follows that the complex $P\otimes_R R/\frakm$ has zero differentials. Hence $H^i({E\dotimes_R R/\frakm}) = H^i({P\otimes_R R/\frakm}) = {P^i \otimes_R R/\frakm}$. In particular, if $E\dotimes_R R/\frakm = 0$ then ${E=0}$. For a nonzero such complex, right-exactness at the top degree~$m$ gives ${H^m(P)\otimes_R R/\frakm} = H^m({P\otimes_R R/\frakm}) = {P^m \otimes_R R/\frakm}$ which is nonzero, so \eqref{eq:top-homology} holds.
\end{Rem}

\begin{Rem}\label{rem:noetherian-pseudo}
	For a noetherian scheme $X$, the pseudo-coherent complexes are precisely those complexes whose cohomology modules are coherent and bounded on the right:\footnote{The author has a strong personal preference for saying bounded on the left and bounded on the right rather than bounded below and bounded above. The former terminology is superior because it is agnostic about whether you are using homological or cohomological grading for complexes. The author was impressed by the use of $\sqsubset$ and $\sqsupset$ in \cite[Appendix~A]{Christensen00}. It is far clearer than using $+$ and $-$. Not only do the latter symbols need to be swapped depending on the grading convention, but their meaning is also far from obvious.}
	\begin{equation}\label{eq:Dps-as-Dpcoh}
		\Dps(X) = \Dpcoh(X).
	\end{equation}
	Moreover, the bounded coherent complexes are precisely the pseudo-coherent complexes that are bounded on the left:
	\begin{equation}\label{eq:Dbcoh-as-Dpsb}
		\Dbcoh(X) = \Dps(X) \cap \Dminus(X).
	\end{equation}
\end{Rem}

\begin{Rem}\label{rem:globally-bounded}
	More generally, every pseudo-coherent complex over a quasi-compact scheme is bounded on the right, $\Dps(X)\subseteq \Der^{\sqsupset}(X)$, and the rightmost cohomology module is always finitely presented as an $\OX$-module. This follows from \cref{rem:affine-pseudo}.
\end{Rem}

\begin{Rem}\label{rem:pseudo-tt}
	The pseudo-coherent complexes form a tensor-triangulated subcategory
	\[
		\Dps(X) \subseteq \Der(X)
	\]
	and for any morphism $f:X \to Y$, the derived pullback $f^*:\Der(Y)\to\Der(X)$ always preserves pseudo-coherent complexes. In contrast, $\Dbcoh(X)$ is only closed under the tensor-product when $X$ is regular, which is precisely when it coincides with~$\Dperf(X)$. However, we always have $\Dbcoh(X) \otimes \Dperf(X) \subseteq \Dbcoh(X)$.
\end{Rem}

\begin{Rec}\label{rec:perfect-generator}
	Since our schemes are quasi-compact and quasi-separated, the derived category $\Der(X)$ always admits a perfect generator $g\in\Dperf(X)$ by \cite{BondalVandenBergh03}. A useful fact is that the internal hom $\ihom{g,-}$ reflects left-bounded complexes: if $\ihom{g,a} \in \Der^{\sqsubset}(X)$ then $a \in \Der^{\sqsubset}(X)$. This is just a consequence of the fact that $\Der^{\sqsubset}(X)$ is a thick subcategory, so that $\cat C_a \coloneqq \SET{c \in \Dperf(X)}{\ihom{c,a}\in \Der^{\sqsubset}(X)}$ is a thick subcategory. Hence, if $\cat C_a$ contains $g$ then it contains $\unit$. The same argument shows that $\ihom{g,-}$ also reflects right-bounded complexes. Furthermore, if $p:X\to\Spec(\bbZ)$ denotes the structure morphism then it is also true that $p_*\ihom{g,-}:\Der(X) \to \Der(\bbZ)$ reflects left-bounded (respectively, right-bounded) complexes. This is slightly less formal, however; see \cite[Section~\href{https://stacks.math.columbia.edu/tag/0GEI}{0GEI}]{stacks-project}.	
\end{Rec}

\begin{Rec}\label{rem:approximation}
	A pseudo-coherent complex $b\in\Dps(X)$ can be ``approximated'' by perfect complexes as follows. For each $m \in \bbZ$ there is an exact triangle
	\[
		c \to b \to d \to \Sigma c
	\]
\enlargethispage{1\baselineskip}
	with $c \in \Dperf(X)$ and $d \in \Der^{\le m}(X)$. This is proved in \cite[Section~\href{https://stacks.math.columbia.edu/tag/08EL}{08EL}]{stacks-project}.
\end{Rec}

\section{The spectrum of pseudo-coherent complexes}

We may consider the Balmer spectrum $\Spc(\Dps(X))$ of the category of pseudo-coherent complexes. The inclusion $\Dperf(X) \hookrightarrow \Dps(X)$ induces a map
\[
	\varphi:\Spc(\Dps(X)) \to \Spc(\Dperf(X))\cong X
\]
and our present goal is to prove that its fibers are irreducible; see \cref{thm:fibers-irreducible}.

\begin{Def}\label{def:support}
	We can define the support of a pseudo-coherent complex~$E$ by
	\begin{align}
		\Supp(E) &= \SET{x \in X}{E_x \neq 0 \text{ in } \Der(\OXx)}.\label{eq:support}
	\intertext{This is specialization closed and it follows from \cref{rem:pseudo-nak} that we have an equality}
		\Supp(E) &= \SET{x \in X}{E \dotimes_{\OX} k(x) \text{ in } \Der(k(x))}.\label{eq:h-support}
	\end{align}
	This defines a support theory on $\Dps(X)$ which satisfies the detection property and the tensor-product property.\footnote{The detection property comes from the definition in terms of stalks. The tensor-product property comes from the characterization using the residue fields.} Consequently, the formula 
	\[
		\tame(x) \coloneqq \SET{E \in \Dps(X)}{x \not\in\Supp(E)}
	\]
	provides a well-defined function
	\[
		\tame:X \to \Spc(\Dps(X)).
	\]
	See \cite[Theorem~3.2]{Balmer05a} and \cite[Construction~5.12]{SandersZhang26} for further details. This is the unique function such that 
	\begin{equation}\label{eq:tau-inverse}
		\tau^{-1}(\supp(E)) = \Supp(E)
	\end{equation}
	where $\supp(E) \subseteq \Spc(\Dps(X))$ is the universal Balmer support of the pseudo-coherent complex $E \in \Dps(X)$.
\end{Def}

\begin{Rem}
	By construction we have a set-theoretic splitting
	\[\begin{tikzcd}
		X \ar[r,"\tau"'] \ar[rr,bend left=20,"\id"] & \Spc(\Dps(X)) \ar[r,"\varphi"'] & X.
	\end{tikzcd}\]
	We will show that the map $\tame$ provides the generic points of the fibers of $\varphi$.
\end{Rem}

\begin{Lem}\label{lem:koszul}
	Let $R$ be a commutative ring and let $I=(a_1,\ldots,a_n)$ be a finitely generated ideal. The Koszul complex $K\coloneqq \bigotimes_{i=1}^n \cone(a_i)$ is contained in the thick ideal generated by $R/I$.
\end{Lem}

\begin{proof}
	Since the inclusion $\Mod(R) \hookrightarrow \Der(R)$ is lax symmetric monoidal, any \mbox{$R/I$-module} $M$ is a module for the ring object $R/I\in\Der(R)$ when regarded as an object in~$\Der(R)$. Hence the canonical map $M \to R/I\dotimes_R M$ is a split monomorphism in the derived category. In particular, any $R/I$-module is contained in the thick ideal generated by $R/I$. Next recall that for any $a\in R$, the morphism $a:\unit\to\unit$ in $\Der(R)$ is annihilated when tensored with $\cone(a)$. Hence $a_j\otimes K = 0$ for each $j=1,\ldots,n$. It follows that $IH^i(K)=0$ for each $i \in \bbZ$. Thus each $H^i(K)$ is an $R/I$-module. By induction on the length of the Postnikov tower, any bounded complex lies in the thick subcategory generated by its cohomology modules. We conclude that $K$ is contained in the thick ideal of $\Der(R)$ generated by $R/I$.
\end{proof}

\begin{Lem}\label{lem:local-triangle}
	Let $(R,\frakm)$ be a commutative local ring and let $E \in \Dps(R)$ be a nonzero pseudo-coherent complex. There exists an exact triangle
	\[
		W \xrightarrow{\xi} R \xrightarrow{\alpha} \Sigma^m E \to \Sigma W
	\]
	and a nonzero perfect complex $K \in \Dperf(R)$ such that $\xi^{\otimes n}\otimes K = 0$ for some $n \ge 1$.
\end{Lem}

\begin{proof}
	Choose a quasi-isomorphism $P \xrightarrow{\sim} E$ where $P$ is a right bounded complex of finitely generated free modules satisfying $d^i(P^i)\subseteq \frakm P^{i+1}$ for all $i$ as in~\cref{rem:pseudo-nak}. Replacing $E$ by a shift $\Sigma^m E$, we may assume that $P^0 \neq 0$ and $P^i = 0$ for $i>0$. Thus $P^0 \otimes_R R/\frakm$ is a nonzero vector space over $R/\frakm$. We can thus construct a nonzero linear map $P^0 \otimes_R R/\frakm \to R/\frakm$. This lifts to an $R$-linear map $f:P^0\to R$ and we have $f(P^0) \not\subseteq \frakm$ since $f \otimes_R R/\frakm$ is nonzero by construction. Thus $f(P^0)$ contains a unit of $R$. It follows that there exists $u \in P^0$ such that $f(u)=1$. We may then consider the morphism of complexes $\alpha: R \to P$ which in degree zero maps $1\in R$ to $u\in P^0$. Note that
	\[
		I\coloneqq \im(P^{-1} \xrightarrow{d^{-1}} P^0 \xrightarrow{f} R) \subseteq \frakm
	\]
	because $d^{-1} \otimes_R R/\frakm = 0$. Since $P^{-1}$ is finitely generated, the ideal $I$ is finitely generated. Moreover, the map $f\otimes_R R/I : P^0 \otimes_R R/I \to R/I$ defines a morphism of complexes $P\otimes_R R/I \to R/I$ which is split by $\alpha \otimes_R R/I$. We have thus constructed a map $R\to P\simeq E$ in $\Der(R)$ which becomes split monic after tensoring with $R/I$. Its fiber $\xi:W\to R$ thus satisfies $\xi \dotimes_R R/I=0$. It then follows from~\cref{lem:koszul} and~\cref{rec:nilpotence} that~$\xi$ acts tensor-nilpotently on any Koszul complex $K$ of the finitely generated ideal~$I$. Note that $K$ is nonzero since the ideal $I$ is proper.
\end{proof}

\begin{Lem}\label{lem:kill-pseudo-P}
	Let $\alpha:E \to C$ be a morphism from a pseudo-coherent complex $E$ to a perfect complex $C$. Let $\cat P \in \Spc(\Dps(X))$ be any prime ideal and let $x \coloneqq \varphi(\cat P)$ be its image. If $\alpha_x = 0$ in $\Dps(\OXx)$ then $\alpha_{\cat P}=0$ in $\Dps(X)/\cat P$.
\end{Lem}

\begin{proof}
	The localization $\Der(X) \to\Der(\OXx)$ is the finite localization corresponding to the Thomason subset $Y\coloneqq\gen(x)^c \subseteq X \cong\Spc(\Dperf(X))$. Its kernel~$\cat L$ is $\Locidsub{e_Y}$ where $e_Y \to \unit \to f_Y \to \Sigma e_Y$ is the associated idempotent triangle. See \cite{BalmerFavi11} and~\cite{BarthelHeardSanders23a}. Since $\alpha_x = 0$, we have 
	\[\begin{tikzcd}[row sep=2.5ex]
		& E \ar[d,"\alpha"] \ar[dr,bend left=15,"0"] \ar[dl,dashed,bend right=15,"\exists","\beta"'] &\\
		e_Y \otimes C \ar[r] & C \ar[r] & f_Y \otimes C
	\end{tikzcd}\]
	and it suffices to prove that $\beta$ is killed by the quotient $q:\Dps(X) \to \Dps(X)/\cat P$. Let $\beta':\unit \to \ihom{E,e_Y\otimes C}$ denote the adjoint map.

	The support of~\cref{def:support} makes sense for any complex $A$ and is always specialization closed. Hence $\cat L =\SET{A \in \Der(X)}{\Supp(A) \subseteq Y} \eqqcolon \Der_Y(X)$. Moreover, $\Supp(A)=\bigcup_{i\in\bbZ} \Supp(H^i(A))$ coincides with the union of the ordinary supports of the cohomology modules. Thus the inclusion $\cat L =\Der_Y(X) \hookrightarrow \Der(X)$ is $t$-exact for the standard \mbox{$t$-structure} since the truncations just throw away some of the cohomology modules. Its right adjoint $e_Y \otimes - :\Der(X) \to \Der_Y(X)$ is thus left $t$-exact. Hence $e_Y\otimes C$ is left-bounded since~$C$ is perfect. Derived hom from a pseudo-coherent complex to a left-bounded complex commutes with taking stalks by~\cite[Lemma~\href{https://stacks.math.columbia.edu/tag/0GM8}{0GM8}]{stacks-project}. In other words, $f_Y \otimes \ihom{E,e_Y \otimes C} \simeq \ihom{f_Y \otimes E, f_Y \otimes e_Y \otimes c}$ and this vanishes since ${f_Y\otimes e_Y= 0}$. Therefore $\ihom{E,e_Y \otimes C} \in \Locidsub{e_Y}$. Thus, since~$\unit$ is compact, the map $\beta':\unit \to \ihom{E,e_Y\otimes C}$ factors as $\unit \to d\to \ihom{E,e_Y\otimes C}$ for some compact $d\in\Locidsub{e_Y} \cap \Dperf(X)$. The morphism $\beta$ then factors as
	\[
		E \simeq E \otimes \unit \to E \otimes d \to E\otimes \ihom{E,e_Y\otimes C} \xrightarrow{\text{ev}} e_Y \otimes C.
	\]
	Now $d \in \Dperf(X) \cap \cat P$ since $\varphi(\cat P)=x \not\in \Supp(d)$. Hence $E \otimes d \in \cat P$ since $\cat P$ is an ideal in $\Dps(X)$. Thus $q(\beta) = 0$ in $\Dps(X)/\cat P$ and hence $q(\alpha) = 0$ too.
\end{proof}

\begin{Lem}\label{lem:verdier}
	Let $\cat T$ be a compactly generated triangulated category and $\cat L \coloneqq \Loc(\cat G)$ the localizing subcategory generated by a set of compact objects $\cat G \subseteq \cat T^c$. Write $q:\cat T \to \cat T/\cat L \eqqcolon \cat S$ for the associated Verdier localization. Suppose $\alpha: q(x) \to q(t)$ is a morphism in $\cat S$ with $x \in \cat T^c$ compact and $t \in \cat T$ arbitrary. Then there exists a compact object $c \in \cat T^c$ and morphisms $\beta: c \to t$ and $\sigma:c\to x$ such that $q(\sigma)$ is an isomorphism and $q(\beta) = \alpha \circ q(\sigma)$.
\end{Lem}

\begin{proof}
	The morphism $\alpha$ is represented by a fraction $x\xleftarrow{u} y \xrightarrow{f} t$ and we have an exact triangle $y\xrightarrow{u}x\xrightarrow{v} \ell \to \Sigma y$ where $\ell \in \cat L$. Since $x$ is compact, $v$ factors as $x\xrightarrow{w} z \to \ell$ for some $z \in \Locsub{\cat G} \cap \cat T^c = \thicksub{\cat G}$. We then have an exact triangle $c \xrightarrow{\sigma} x \xrightarrow{w} z \to \Sigma c$ of compact objects. Since $v\circ \sigma = 0$, the morphism $\sigma$ factors as $c\xrightarrow{\theta} y \xrightarrow{u} x$ for some morphism $\theta$. Thus, defining $\beta \coloneqq f\circ \theta$, we see that $\alpha$ is represented by the fraction $x \xleftarrow{\sigma} c \xrightarrow{\beta} t$ and the result follows.
\end{proof}

\begin{Thm}\label{thm:fibers-irreducible}
	Let $X$ be a quasi-compact and quasi-separated scheme. The fibers~of
	\[
		\varphi: \Spc(\Dps(X)) \to X
	\]
	are irreducible. The generic point of a fiber $\varphi^{-1}(\{x\})$ is the point $\tau(x)$.
\end{Thm}

\begin{proof}
	Let $\cat P \in \Spc(\Dps(X))$ be a prime ideal with $x\coloneqq \varphi(\cat P) \in X$. We claim that $\cat P \subseteq \tau(x)$ which means that $\cat P$ is in the closure of $\tau(x)$ in the Balmer topology on the spectrum. To this end, let $E \in \cat P$ and suppose for a contradiction that $E\not\in\tau(x)$. By definition, this means that $E_x \neq 0$ in $\Der(\OXx)$. Applying \cref{lem:local-triangle} and replacing~$E$ by a shift without loss of generality, we have an exact triangle
	\[
		w\xrightarrow{\xi} \unit_x \xrightarrow{\alpha} E_x \to \Sigma w
	\]
	in $\Der(\OXx)$ and a nonzero compact $k \in \Der(\OXx)$ such that $\xi^{\otimes n} \otimes k = 0$ for some~$n \ge 1$.

	By~\cref{lem:verdier}, there is a compact object $c \in \Dperf(X)$ and morphisms $\beta:c\to E$ and $\sigma:c \to \unit$ in $\Der(X)$ with $\sigma_x$ an isomorphism and $\beta_x=\alpha \circ \sigma_x$. Choosing an exact triangle $W \xrightarrow{\zeta} c \xrightarrow{\beta} E \to \Sigma W$ in $\Der(X)$, we obtain a commutative diagram
	\[\begin{tikzcd}[ampersand replacement=\&]
		W_x \ar[r,"\zeta_x"] \ar[d,"\simeq"'] \& c_x \ar[r,"\beta_x"] \ar[d,"\sigma_x","\simeq"'] \& E_x \ar[d,equals] \ar[r] \& \Sigma W_x  \ar[d,"\simeq"]\\
		w \ar[r,"\xi"] \& \unit_x \ar[r,"\alpha"] \& E_x \ar[r] \& \Sigma w.
	\end{tikzcd}\]
	There exists a compact object $K \in \Dperf(X)$ with $K_x \simeq k \oplus \Sigma k$. We thus see that 
	\begin{equation}\label{eq:zeta-nilpotent}
		\zeta^{\otimes n}\otimes K : W^{\otimes n}\otimes K \to c^{\otimes n}\otimes K
	\end{equation}
	is a morphism in $\Der(X)$ from a pseudo-coherent complex to a perfect complex which vanishes in $\Der(\OXx)$. By~\cref{lem:kill-pseudo-P}, it also vanishes in the quotient $q:\Dps(X)\to\Dps(X)/\cat P$. On the other hand, $E\in \cat P$ by hypothesis, so $q(\zeta)$ is an isomorphism. Since $q(\zeta^{\otimes n} \otimes K)$ is both zero and an isomorphism, the object $q(c^{\otimes n} \otimes K)$ vanishes, which means $c^{\otimes n} \otimes K \in \cat P$. Hence either $K \in \cat P$ or $c \in \cat P$ since $\cat P$ is prime. But these are compact objects and $\cat P \cap \Dperf(X) = \SET{b \in \Dperf(X)}{ b_x = 0}$. Thus $K_x = 0$ or $c_x = 0$. This is impossible since $K_x$ contains the nonzero summand $k$ while $c_x \simeq \unit_x \neq 0$.
\end{proof}

\begin{Rem}
	In the case of a noetherian affine scheme, \cref{thm:fibers-irreducible} is due to Matsui--Takahashi~\cite{MatsuiTakahashi17}. The theorem does not readily reduce to the affine case since for a quasi-compact open $U \subset X$, the commutative square
	\[\begin{tikzcd}
			\Spc(\Dps(U)) \ar[r]\ar[d] & \Spc(\Dps(X))\ar[d] \\
			U \ar[r,hook] & X
	\end{tikzcd}\]
	is not a pullback \cite{SandersZhang26}. This is the case even for $U=\Spec(\bbQ) \hookrightarrow \Spec(\bbZ_{(p)})=X$.
\end{Rem}

\section{Reflecting pseudo-coherence}

The main result of the last section has the following direct consequence:

\begin{Prop}\label{prop:virtually-small}
	Let $X$ be a quasi-compact and quasi-separated scheme. Suppose $a \in \Dps(X)$ and $c \in \Dperf(X)$ satisfy $\Supp(c) \subseteq \Supp(a)$. Then there exists an object ${b \in \Dps(X)}$ and an integer $n\ge 1$ such that $c^{\otimes n} \in \thicksub{a \otimes b}$. 
\end{Prop}

\begin{proof}
	The universal Balmer support of a perfect complex $c\in\Dperf(X)$ coincides with $\Supp(c)$ under the identification $\Spc(\Dperf(X)) \cong X$. Hence $\varphi^{-1}(\Supp(c))=\supp(c)$ is the universal support of $c$ regarded as an object of $\Dps(X)$. The hypothesis $\Supp(c) \subseteq \Supp(a)$ thus implies
	\[
		\supp(c) = \varphi^{-1}(\Supp(c)) \subseteq \varphi^{-1}(\Supp(a)) = \varphi^{-1}(\tau^{-1}(\supp(a))).
	\]
	Thus, if $\cat P \in \supp(c)$ then $\tau(\varphi(\cat P)) \in \supp(a)$. But $\cat P$ is in the closure of $\tau(\varphi(\cat P))$ by \cref{thm:fibers-irreducible}. Since $\supp(a)$ is closed, it follows that $\cat P \in \supp(a)$. This establishes that $\supp(c) \subseteq \supp(a)$. By the classification of radical thick ideals, $c \in \sqrt{\thickidsub{a}}$, so that $c^{\otimes n} \in \thickidsub{a} = \thicksub{a \otimes \Dps(X)}$ for some $n\ge 1$. It follows that~$c^{\otimes n}$ is in the thick subcategory generated by a finite collection of objects $a \otimes b_1, \ldots a\otimes b_k$ with $b_1,\ldots,b_k \in \Dps(X)$. We may then take $b \coloneqq \bigoplus_{1\le i\le k} b_i$.
\end{proof}

\begin{Exa}\label{exa:full}
	If $a$ is a pseudo-coherent complex with $\Supp(a) = X$ then there is a pseudo-coherent complex $b$ such that $\unit \in \thicksub{a \otimes b}$.
\end{Exa}

\begin{Exa}\label{exa:nonempty}
	If $a$ is a nonzero pseudo-coherent complex then there exists a pseudo-coherent complex $b$ such that $\thicksub{a \otimes b}$ contains a nonzero perfect complex. This follows from~\cref{prop:virtually-small} since the support $\Supp(a) = \bigcup_{i \in \bbZ} \Supp(H^i(a))$ always contains a nonempty Thomason closed subset. This is because, as recalled in \cref{rem:globally-bounded}, the complex is bounded on the right and the rightmost cohomology module~$H^m(a)$ is finitely presented. Over a quasi-compact scheme, the support of a finitely presented $\OX$-module is always Thomason closed. Thus $\Supp(H^m(a))=\Supp(c)$ for some nonzero perfect complex~$c$.
\end{Exa}

\begin{Rem}\label{rem:virtually-small}
	Let $R$ be a commutative ring. A nonzero complex $a \in \Der(R)$ is said to be \emph{virtually small} in the terminology of~\cite{DwyerGreenleesIyengar06} if the thick subcategory $\thicksub{a}$ contains a nonzero perfect complex. Pseudo-coherent complexes need not be virtually small in general. Indeed, Pollitz \cite[Theorem~5.2]{Pollitz19} establishes that a noetherian local ring must be a complete intersection if every pseudo-coherent complex is virtually small. Nevertheless, \cref{exa:nonempty} shows that every pseudo-coherent complex is virtually small up to tensoring by a pseudo-coherent complex.
\end{Rem}

\begin{Exa}\label{exa:key-example}
	Let $k$ be a field and consider $A=k[x,y]/(x^2,xy,y^2)$ and $B=k[x]/(x^2)$.
	Conceptually, $A$ is the trivial square-zero extension
	\[
		A\simeq B \ltimes  k
	\]
	where $k$ is regarded as a $B$-module via $B \to B/(x) = k$. As an $A$-module $B$ is finitely generated and hence coherent (since $A$ is noetherian) but it is not perfect. Otherwise, the exact sequence $0\to (y) \to A \to B \to 0$ would imply that $(y) \simeq k$ is perfect, but $A$ is not regular. Consider the complex
	\[
		E = (\cdots \xrightarrow{x} A \xrightarrow{x} A \xrightarrow{x} A \to 0) \in \Der(A)
	\]
	where the rightmost $A$ is in cohomological degree zero. We have $H^0(E) = A/(x) \neq 0$ and $H^i(E) = \ker(x)/\im(x) = (x,y)/(x) \simeq k \neq 0$ for $i < 0$. Thus $E$ is a pseudo-coherent complex which is not bounded coherent. On the other hand, now consider
	\[
		B\dotimes_A E = (\cdots \xrightarrow{x} B \xrightarrow{x} B \xrightarrow{x} B \to 0) \in \Der(B).
	\]
	We have $H^0(B\dotimes_A E) = B/(x) = k$ but $H^i(B\dotimes_A E) = \ker(x)/\im(x) = 0$ for $i<0$. Hence we have an isomorphism
	\[
		B\dotimes_A E \simeq k
	\]
	in $\Der(B)$ and hence also in $\Der(A)$ by restriction of scalars.
\end{Exa}

\begin{Rem}
	The pseudo-coherent twist~$b\in\Dps(X)$ in \cref{prop:virtually-small} cannot be chosen to be perfect in general. Indeed, in the above example $B$ is pseudo-coherent as an $A$-module. However suppose there is a nonzero perfect complex $c\in\Dperf(A)$ such that $c \in \thicksub{B \dotimes_A b}$ for some perfect complex $b\in\Dperf(A)$. Then 
	\begin{align*}
		E \dotimes_A c \in \thicksub{E \dotimes_A B \dotimes_A b} = \thicksub{k\dotimes_A b}.
	\end{align*}
	But $k \in \Dbcoh(A)$ and hence $k \dotimes_A b \in \Dbcoh(A)$ too since $b$ is perfect. It follows that $E \dotimes_A c \in \Dbcoh(A)$. Now $\cat C \coloneqq \SET{ c' \in \Dperf(A) }{ E\dotimes_A c' \in \Dbcoh(A)}$ is a thick ideal of $\Dperf(A)$. However, since $\Spec(A)=\Spc(\Dperf(A))$ is just a single point, every nonzero perfect complex generates $\Dperf(A)$ as a radical ideal. Since $0\neq c \in \cat C$, it would follow that $\unit \in \cat C$ meaning $E \in \Dbcoh(A)$ which is false.
\end{Rem}

\begin{Prop}\label{prop:reflects-pseudo}
	Let $f:X \to Y$ be a surjective proper morphism of noetherian schemes. The derived pullback $f^*:\Der(Y)\to\Der(X)$ reflects pseudocoherent complexes: If $a \in \Der(Y)$ satisfies $f^*(a) \in \Dps(X)$ then $a \in \Dps(Y)$.
\end{Prop}

\begin{proof}
	Since derived categories of noetherian schemes are stratified, the support~\eqref{eq:h-support} of \cref{def:support} coincides with the Balmer--Favi support by \cite[Theorem~4.7]{BarthelHeardSanders23b}. It then follows from \cite[Corollary~13.15]{BarthelCastellanaHeardSanders23app} that the pseudo-coherent complex~$f_*(\unit_X)$ has full support: $\Supp(f_*(\unit_X)) = Y$. Hence, by \cref{prop:virtually-small}, there is a pseudo-coherent complex $b \in \Dps(Y)$ such that 
	\begin{equation}\label{eq:descendable}
		\unit_Y \in \thicksub{f_*(\unit_X)\otimes b}.
	\end{equation}
	Now let $a \in \Der(Y)$. It follows from \eqref{eq:descendable} that 
	\[
		a \in \thicksub{a \otimes f_*(\unit_X) \otimes b} = \thicksub{f_*f^*(a) \otimes b}.
	\]
	If $f^*(a) \in \Dps(X)$ then the right-hand side is contained in $\Dps(Y)$ by \cref{thm:grothendieck} and \cref{rem:pseudo-tt}.
\end{proof}

\begin{Rem}\label{rem:no-reflect-bounded}
	The derived pullback of a proper surjective morphism need not reflect bounded coherent complexes. Indeed, recall \cref{exa:key-example}. The morphism ${f:\Spec(B) \to \Spec(A)}$ is finite hence proper and is trivially surjective since the spaces are single points. The complex $E$ is a pseudo-coherent complex of $A$-modules such that $f^*(E) = k$ is bounded coherent, but $E$ itself is not bounded coherent.
\end{Rem}

\begin{Rem}
	\Cref{prop:reflects-pseudo} is false if the surjectivity hypothesis is dropped, since then~$f^*$ could even annihilate some perfect complexes. The statement is also false if properness is dropped. For example, the $\bbZ$-module $\bbZ[1/p]$ is not finitely generated, but it becomes finitely generated after base change along $\bbZ \to \bbZ[1/p]\times \bbZ/p$.
\end{Rem}

\begin{Exa}\label{exa:no-reflect}
	Let $f:\Pone \to \Spec(k)$ be the structure map of the projective line. Note that $f_*(\cat O(-1))=0$ since $H^i(\Pone,\cat O(-1)) = 0$ for all $i \in \bbZ$. The complex $\bigoplus_{n \ge 0} \Sigma^{-n}\cat O(-1)$ is thus annihilated by $f_*$ but it is not pseudo-coherent. In particular, pseudo-coherent complexes need not be reflected by~$f_*$.
\end{Exa}

\section{Characterizing pseudo-coherence}

We remind the reader of a basic tool in algebraic geometry which shows that proper morphisms are close to being projective.

\begin{Thm}[Chow's Lemma]\label{thm:chow}
	Let $f:X \to Y$ be a proper morphism of noetherian schemes. Then there exists a surjective proper morphism $g:X'\to X$ such that $f\circ g:X' \to Y$ is projective.
\end{Thm}

\begin{proof}
	See \cite[Theorem~5.6.1]{EGAII} or \cite[Lemma~\href{https://stacks.math.columbia.edu/tag/0200}{0200}]{stacks-project}. It is also mentioned in \cite[Exercise~4.10]{Hartshorne77}. The origin of the result is \cite[Lemma~1]{Chow57}.
\end{proof}

\begin{Thm}\label{thm:proper-pseudo}
	Let $f:X\to Y$ be a proper morphism of noetherian schemes. Then
	\begin{equation}\label{eq:Dps-char}
		\Dps(X) = \SET{a \in \Der(X)}{f_*(a \otimes b)\in \Dps(Y) \text{ for all } b\in \Dps(X)}.
	\end{equation}
\end{Thm}

\begin{proof}
	The inclusion $\subseteq$ follows from the Serre--Grothendieck theorem (\cref{thm:grothendieck}) since $\Dps(X) \otimes \Dps(X) \subseteq \Dps(X)$. See~\cite[Th\'{e}or\`{e}me 2.2.1]{SGA6-expose-III}. We prove the~$\supseteq$ inclusion. First we observe that the collection of proper morphisms of noetherian schemes for which~\eqref{eq:Dps-char} holds is closed under composition. Indeed, suppose it holds for two morphisms $f:X\to Y$ and $g:Y\to Z$. Consider any $a \in \Der(X)$ with the property that $g_*f_*(a \otimes b)\in\Dps(Z)$ for all $b \in \Dps(X)$. Then for any $c\in \Dps(X)$ and $d \in \Dps(Y)$ we have
	\[
		g_*(f_*(a\otimes c) \otimes d) \simeq g_*f_*(a \otimes c \otimes f^*(d)) \in \Dps(Z).
	\]
	The result for $g$ then implies $f_*(a \otimes c) \in \Dps(Y)$ and then the result for $f$ implies that $a \in \Dps(X)$.

	We now prove the theorem for a projective morphism $f:X\to Y$ bearing in mind the above composition fact. The argument for this case appears in \cite[Theorem~5.21(b)]{BalmerDellAmbrogioSanders16}. Write $f$ as the composite $X \hookrightarrow \mathbb{P}^n_Y \to Y$ of a closed immersion followed by the standard projection. The case of a closed immersion (of a noetherian scheme) is standard; see~\cite[Lemma~\href{https://stacks.math.columbia.edu/tag/09VA}{09VA}]{stacks-project}.

	Now consider $f:\mathbb{P}^n_Y \to Y$. Let $\Delta : \PnY \to \PnY \times_Y \PnY$ denote the diagonal and let $p,q:\PnY \times_Y \PnY \to \PnY$ denote the two projections. The Beilinson resolution of the diagonal  from \cite{Beilinson78} and \cite[Lecture~3]{Caldararu05} holds for projective space over any base scheme. It provides a resolution of $\cat O_\Delta \in \Dperf(\PnY\times_Y \PnY)$:
	\[
		0 \to p^*(c_n)\otimes q^*(d_n) \to p^*(c_{n-1}) \otimes q^*(d_{n-1}) \to \cdots \to p^*(c_0)\otimes q^*(d_0) \to \cat O_\Delta \to 0
	\]
	where $c_i = \cat O_{\PnY}(-i)$ and $d_i = \Omega_{\PnY/Y}^i(i)$. Any bounded complex is contained in the thick subcategory generated by its individual terms, by inductively using brutal truncations. Thus $\cat O_{\Delta}$ is contained in the thick subcategory generated by the $p^*(c_i)\otimes q^*(d_i)$. Note that the objects $c_i, d_i \in \Der(\PnY)$ are vector bundles hence, in particular, perfect complexes. Next observe that we can express any $a \in \Der(\PnY)$ as
	\[
		a = q_*\Delta_*(a) = q_*\Delta_*\Delta^*p^*(a) = q_*(p^*(a)\otimes\cat O_{\Delta} ).
	\]
	Using the projection formula, we obtain
	\[
		a \in \thicksub{d_i \otimes q_*(p^*(a\otimes c_i))}.
	\]
	Moreover, since $f:\PnY \to Y$ is flat, we have $f^*f_* \simeq q_*p^*$ by applying flat base-change to the Cartesian square defining $\PnY \times_Y \PnY$. We have thus established that
	\begin{equation}\label{eq:a-thick-Pn}
		a \in \thicksub{d_i \otimes f^*f_*(a\otimes c_i)}
	\end{equation}
	for any $a \in \Der(\PnY)$. Hence, if $a$ has the property that $f_*(a\otimes c)\in\Dps(Y)$ for all pseudo-coherent $c$ then, since the $c_i$ and $d_i$ are perfect and $f^*$ preserves pseudo-coherent complexes, the right-hand side of~\eqref{eq:a-thick-Pn} is contained in $\Dps(X)$. This establishes the theorem for projective morphisms.

	Now consider an arbitrary proper morphism $f:X\to Y$. By Chow's Lemma there exists a surjective proper morphism $g:X'\to X$ such that $f\circ g:X' \to Y$ is projective. Let $a \in \Der(X)$ and suppose $f_*(a\otimes b) \in \Dps(Y)$ for all $b \in \Dps(X)$. We claim that $a \in \Dps(X)$. By \cref{prop:reflects-pseudo}, it suffices to prove that $g^*(a) \in \Dps(X')$. By the projective case already established, this holds if $f_*g_*(g^*(a) \otimes c) \in \Dps(Y)$ for all $c \in \Dps(X')$. By the projection formula,
	\[
		f_*g_*(g^*(a) \otimes c)\simeq f_*(a \otimes g_*(c))
	\]
	which is contained in $\Dps(Y)$ by hypothesis since $b\coloneqq g_*(c) \in \Dps(X)$ because~$g$ is proper.
\end{proof}

\begin{Rem}
	The following lemma allows us to rewrite~\eqref{eq:Dps-char} as
	\begin{equation*}
		\Dps(X) = \SET{a \in \Der(X)}{f_*(a \otimes b)\in \Dps(Y) \text{ for all } b\in \Dperf(X)}.
	\end{equation*}
	In other words, it makes no difference whether we twist by all pseudo-coherent complexes or only by perfect complexes.
\end{Rem}

\begin{Lem}\label{lem:whatever-twist}
	Let $f:X\to Y$ be a morphism of noetherian schemes and let $a \in \Der(X)$. The following statements are equivalent:
	\begin{enumerate}
		\item $f_*(a \otimes b) \in \Dps(Y)$ for all $b \in \Dps(X)$.
		\item $f_*(a \otimes b) \in \Dps(Y)$ for all $b \in \Dperf(X)$.
	\end{enumerate}
\end{Lem}

\begin{proof}
	The nontrivial implication is $(b) \Rightarrow (a)$. Let $F:\Der(X) \to \Der(Y)$ denote the functor $F(-) \coloneqq f_*(a\otimes -)$ with right adjoint $G(-)\coloneqq \ihom{a,f^!(-)}$. Suppose $F(c) \in \Dps(Y)$ for all $c \in \Dperf(X)$. Our goal is to prove that $F(b) \in \Dps(Y)$ for all $b \in \Dps(X)$.

	Recall from~\cref{conv:qcqs} that $\Der(Y)$ denotes $\Derqc(Y)$. However, $\Derqc(Y) \cong \Der(\QCoh(Y))$ since $Y$ is noetherian by~\cite[Proposition~\href{https://stacks.math.columbia.edu/tag/09T4}{09T4}]{stacks-project}. The Grothendieck abelian category~$\QCoh(Y)$ has an injective cogenerator $I \in \QCoh(Y)$ which we will regard as an object in the heart of $\Der(\QCoh(Y))$. Consider the object
	\[
		d\coloneqq G(I) = \ihom{a,f^!(I)} \in \Der(X)
	\]
	and choose a perfect generator $g\in\Dperf(X)$. By hypothesis, $F(g)$ is pseudo-coherent and hence, in particular, bounded on the right (\cref{rem:noetherian-pseudo}). It follows that $\ihom{F(g),I}$ is bounded on the left, since $I$ is injective. Now recall from \cref{rec:perfect-generator} that $(p_X)_*\ihom{g,-}:\Der(X) \to \Der(\bbZ)$ reflects left-bounded complexes, where $p_X:X\to\Spec(\bbZ)$ is the structure morphism. One readily checks from the definitions and the internal adjunction $f_*\ihom{-,f^!(-)}= \ihom{f_*(-),-}$ that
	\[
		(p_X)_*\ihom{g,d} = (p_Y)_*f_*\ihom{g,d} = (p_Y)_*f_*\ihom{F(g),I}
	\]
	and the right-hand side is left-bounded since $\ihom{F(g),I}$ is left-bounded. We conclude that $d$ is left-bounded, say $d \in \Der^{\ge k}(X)$.

	Now consider any right-bounded complex $e \in \Der^{\le n}(X)$. If $j > n-k$ then $\Sigma^{-j} d\in \Der^{\ge k+j}(X) \subseteq \Der^{\ge n+1}(X)$. From the definitions, we have
	\[
		\Hom_{\Der(Y)}(F(e),\Sigma^{-j}I) \simeq \Hom_{\Der(X)}(e,\Sigma^{-j}d)=0.
	\]
	On the other hand, since $I$ is injective,
	\[
		\Hom_{\Der(Y)}(F(e),\Sigma^{-j}I) = H^{-j}\Hom_{\QCoh(Y)}^\bullet(F(e),I) = \Hom_{\QCoh(Y)}(H^j(F(e)),I).
	\]
	Since $I$ is a cogenerator, this implies that $H^j(F(e)) = 0$. This establishes that 
	\begin{equation}\label{eq:F-bound}
		F(\Der^{\le n}(X)) \subseteq \Der^{\le n-k}(Y).
	\end{equation}

	Now consider a pseudo-coherent complex $b \in \Dps(X)$. Our goal is to prove that $F(b) \in \Dps(X)$. Recall from~\cref{rem:approximation} that for any integer $m$ there exists an exact triangle $c_m \to b \to d_m \to \Sigma c_m$ with $c_m \in \Dperf(X)$ and $d_m \in \Der^{\le m}(X)$. We then obtain an exact triangle
	\begin{equation}\label{eq:F-triangle}
		F(c_m) \to F(b) \to F(d_m) \to \Sigma F(c_m).
	\end{equation}
	Our hypothesis implies that $F(c_m)$ is pseudo-coherent since $c_m$ is perfect, while $F(d_m) \in \Der^{\le m-k}(Y)$ by \eqref{eq:F-bound}. Any choice of $m$ thus shows that $F(b)$ is right-bounded. It remains to show that the cohomology modules of $F(b)$ are coherent. Fix any $i\in \bbZ$ and chose a small enough~$m$ such that $m-k<i-1$. Then $H^{i-1}(F(d_m)) = H^i(F(d_m)) = 0$ and the long exact sequence for~\eqref{eq:F-triangle} gives $H^i(F(c_m)) \simeq H^i(F(b))$, which is coherent since $F(c_m)$ is pseudo-coherent. We conclude that $F(b)$ is pseudo-coherent, which completes the proof.
\end{proof}

\begin{Lem}\label{lem:open-quasi-proper}
	Let $j:U \hookrightarrow X$ be a dense open immersion of quasi-compact and quasi-separated schemes. If $j_*:\Der(U) \to \Der(X)$ preserves pseudo-coherent complexes then $U=X$.
\end{Lem}

\begin{proof}
	The image of the map on spectra ${U =\Spc(\Dperf(U)) \to \Spc(\Dperf(X))=X}$ is the homological support of $j_*(\unit_U)$ by \cite[Theorem~5.12]{Balmer20_bigsupport} and \cite[Corollary~5.11]{Balmer20_nilpotence} which is given by~\eqref{eq:h-support}. Since $j_*(\unit_U)$ is pseudo-coherent by hypothesis, this coincides with~\eqref{eq:support} which is specialization closed. Thus, the dense open $U=\Supp(j_*(\unit_U))$ is specialization closed, which implies that $U=X$. Indeed every point is a specialization of a generic point of an irreducible component, and $U$ contains these generic points since it is dense.
\end{proof}

\begin{Thm}[Nagata's Compactification Theorem]\label{thm:nagata}
	Any separated finite-type morphism $f:X\to Y$ of quasi-compact and quasi-separated schemes factors as a quasi-compact open immersion $j:X \hookrightarrow X'$ followed by a proper morphism $p:X'\to Y$.
\end{Thm}

\begin{proof}
	This is established by \cite[Theorem~4.1]{Conrad07}; see also \cite{Conrad07erratum}. It is a generalization of Nagata's theorem \cite[\S 4~Theorem~2]{Nagata63} from noetherian schemes to quasi-compact and quasi-separated schemes. Conrad attributes the proof to unpublished work of Deligne. See also~\cite[Section~\href{https://stacks.math.columbia.edu/tag/0F3T}{0F3T}]{stacks-project}.
\end{proof}

We can now prove the main theorem:

\begin{Thm}\label{thm:main-body}
	Let $f:X \to Y$ be a finite-type separated morphism of noetherian schemes and let $f_*:\Der(X)\to\Der(Y)$ denote the derived pushforward. The following are equivalent:
	\begin{enumerate}
		\item $f$ is a proper morphism;
		\item $\Dps(X) = \SET{a \in \Der(X)}{f_*(a \otimes b)\in \Dps(Y) \text{ for all } b\in \Dperf(X)}$.
		\item $f_*$ preserves pseudo-coherent complexes;
	\end{enumerate}
\end{Thm}

\begin{proof}
	The $(a)\Rightarrow (b)$ implication is \cref{thm:proper-pseudo} together with \cref{lem:whatever-twist} and the $(b)\Rightarrow (c)$ implication is immediate. It remains to prove $(c) \Rightarrow (a)$. By Nagata's Compactification \cref{thm:nagata}, we may factor $f$ as a quasi-compact open immersion $j:X \hookrightarrow X'$ followed by a proper morphism $p:X'\to Y$. Moreover, we may take~$j$ to be schematically dense; see \cite[Remark~4.2]{Conrad07}. Since $Y$ (and hence $X'$) is noetherian, $j$ is dense in the usual topological sense; see \cite[Lemma~\href{https://stacks.math.columbia.edu/tag/083P}{083P}]{stacks-project}. We will show that $f$ is proper by showing that $X=X'$. For this we will show that $j_*$ preserves pseudo-coherence and invoke~\cref{lem:open-quasi-proper}. Thus, suppose $a \in \Dps(X)$. We claim that $j_*(a) \in \Dps(X')$. For any $b \in \Dps(X')$ observe that
	\[
		p_*(j_*(a) \otimes b) \simeq p_*(j_*(a \otimes j^*(b)))\simeq f_*(a\otimes j^*(b))
	\]
	which is contained in $\Dps(Y)$ by the hypothesis~$(c)$ since $a\otimes j^*(b) \in \Dps(X)$. Since~$p$ is proper, we conclude that $j_*(a) \in \Dps(X')$ by~\cref{thm:proper-pseudo}, as desired.
\end{proof}

The theorem can be augmented with analogous criteria involving bounded coherent complexes, as we now explain.

\begin{Lem}\label{lem:preserves-iff}
	Let $f:X\to Y$ be a morphism of noetherian schemes. The derived pushforward $f_*:\Der(X)\to\Der(Y)$ preserves pseudo-coherent complexes if and only if it preserves bounded coherent complexes.
\end{Lem}

\begin{proof}
	Recall that $f_*(\Der^{\le 0}(X)) \subseteq \Der^{\le d}(Y)$ for some $d\ge 0$. It then follows from~\eqref{eq:Dbcoh-as-Dpsb} that if $f_*$ preserves pseudo-coherent complexes then it preserves bounded coherent complexes. Conversely, suppose $f_*$ preserves bounded coherent complexes and let $a \in \Dps(X)$. The right $t$-boundedness of $f_*$ and \eqref{eq:Dps-as-Dpcoh} reduces the question to showing that $H^i(f_*(a))$ is coherent for all $i$. Choose any $k < i -d -1$ and consider the truncation $\tau^{\le k} a \to a \to \tau^{>k} a \to \Sigma \tau^{\le k}a$. From the associated long exact sequence we obtain $H^i(f_*(a)) \simeq H^i(f_*(\tau^{>k}a))$ which is coherent since $f_*(\tau^{>k}a)$ is a bounded coherent complex by hypothesis.
\end{proof}

\begin{Thm}\label{thm:dbcoh-version}
	Let $f:X \to Y$ be a finite-type separated morphism of noetherian schemes. The following are equivalent:
	\begin{enumerate}
		\item $f$ is a proper morphism;
		\item $\Dbcoh(X) = \SET{a \in \Der(X)}{f_*(a \otimes b)\in \Dbcoh(Y) \text{ for all } b\in \Dperf(X)}$.
		\item $f_*$ preserves bounded coherent complexes.
	\end{enumerate}
\end{Thm}

\begin{proof}
	The implication $(b)\Rightarrow (c)$ is immediate and the implication $(c) \Rightarrow (a)$ is equivalent, by \cref{lem:preserves-iff}, to the corresponding implication of \cref{thm:main-body}. It remains to establish $(a) \Rightarrow (b)$. The inclusion $\subseteq$ is provided by \cref{thm:grothendieck} bearing in mind \cref{lem:preserves-iff} and the fact that $\Dbcoh(X) \otimes \Dperf(X) \subseteq \Dbcoh(X)$. For the~$\supseteq$ inclusion, consider any complex $a \in \Der(X)$ satisfying $f_*(a\otimes b) \in \Dbcoh(Y)$ for all $b \in \Dperf(X)$. \Cref{thm:proper-pseudo} and \cref{lem:whatever-twist} imply that $a$ is pseudo-coherent, which means $a \in \Dpcoh(X)$ since $X$ is noetherian. Let $g \in \Dperf(X)$ be a perfect generator. By hypothesis, $f_*(\ihom{g,a})=f_*(g^\vee \otimes a)$ is pseudo-coherent and hence left-bounded. It follows that $(p_Y)_*f_*(\ihom{g,a})=(p_X)_*(\ihom{g,a})$ is left-bounded where $p_X$ and~$p_Y$ are the structure morphisms to $\Spec(\bbZ)$. This implies that $a$ is left bounded by \cref{rec:perfect-generator} and hence $a \in \Dbcoh(X)$.
\end{proof}

\bibliographystyle{alpha}%
\bibliography{bibliography}

\section*{Acknowledgements}
{\scriptsize
Let me share an amusing personal anecdote concerning this paper. I proved the equivalence of $(a)$ and~$(b)$ in \cref{thm:main} about ten years ago and I was really excited because I could not find any reference in the literature to known converses to the Serre--Grothendieck finiteness theorem. In fact, papers in the 1980s, such as \cite{Vella86}, explicitly mentioned that no converse was known. I thought it was really cool that I had found an unknown converse to such a fundamental theorem from the 1950s. Unfortunately for me, Amnon Neeman informed me that the result was in fact due to Joseph Lipman. That was the first blow. 

Then, on my first visit to Oakland in California my car was broken into. All that the thief obtained for their efforts was a canvas shopping bag filled to the brim with my research notes, together with a copy of \cite{Vella86}, borrowed from the library. While one can of course reconstruct a lost theorem, two blows in quick succession were too much and the result was submerged in the wash of life. After revisiting the theorem and adding condition $(c)$, I decided to resurrect the project. Fortunately the thief never wrote up the theorem! 

With all the above said, I would like to thank Amnon Neeman for setting me straight concerning the result's history. I'm also grateful to the UC Santa Cruz library for not making me pay for the stolen copy of \cite{Vella86}. Finally, I would like to thank Paul Balmer and Ivo Dell'Ambrogio for discussions concerning Grothendieck duality and possible ways of characterizing bounded complexes of coherent sheaves. In particular, \cite[Theorem~5.21]{BalmerDellAmbrogioSanders16} was the precursor to some of these ideas.

}

\end{document}